\documentclass{amsart}
\usepackage{easywncy}
\newtheorem{thm}{Theorem}[section]
\newtheorem{prop}[thm]{Proposition}
\theoremstyle{definition}
\newtheorem{defn}[thm]{Definition}
\theoremstyle{remark}
\newtheorem{rem}[thm]{Remark}
\newcommand{\II}{\mathcal{I}}
\newcommand{\Q}{\mathbb{Q}}
\newcommand{\R}{\mathbb{R}}
\newcommand{\sha}{\mathbin{\mathcyr{sh}}}
\newcommand{\ve}[1]{\boldsymbol{#1}}
\begin{document}
\title[Polynomial generalization of the regularization theorem for MZVs]{Polynomial generalization of the regularization theorem for multiple zeta values}
\author{Minoru Hirose}
\address{Faculty of Mathematics, Kyushu University,
 744, Motooka, Nishi-ku, Fukuoka, 819-0395, Japan}
\email{m-hirose@math.kyushu-u.ac.jp}
\author{Hideki Murahara}
\address{Nakamura Gakuen University Graduate School,
 5-7-1, Befu, Jonan-ku, Fukuoka, 814-0198, Japan} 
\email{hmurahara@nakamura-u.ac.jp}
\author{Shingo Saito}
\address{Faculty of Arts and Science, Kyushu University,
 744, Motooka, Nishi-ku, Fukuoka, 819-0395, Japan}
\email{ssaito@artsci.kyushu-u.ac.jp}
\keywords{multiple zeta value; regularization theorem; symmetric multiple zeta value; symmetrized multiple zeta value; finite real multiple zeta value}
\subjclass[2010]{Primary 11M32; Secondary 05A19}
\begin{abstract}
 Ihara, Kaneko, and Zagier defined two regularizations of multiple zeta values and proved the regularization theorem that describes the relation between those regularizations.
 We show that the regularization theorem can be generalized to polynomials
 whose coefficients are regularizations of multiple zeta values and that specialize to symmetric multiple zeta values defined by Kaneko and Zagier.
\end{abstract}
\maketitle

\section{Introduction}
\subsection{Multiple zeta values and two products}
An \emph{index} is a finite (possibly empty) sequence of positive integers.
We denote by $I$ the set of all indices and by $\II$ the $\Q$-linear space spanned by the indices.
An index is said to be \emph{admissible} if either it is empty or its first component is greater than $1$.

If $\ve{k}=(k_1,\dots,k_r)$ is an admissible index, then we define the \emph{multiple zeta value} $\zeta(\ve{k})$ by
\[
 \zeta(\ve{k})=\sum_{m_1>\dots>m_r\ge1}\frac{1}{m_1^{k_1}\dotsm m_r^{k_r}}\in\R;
\]
we understand that $\zeta(\emptyset)=1$.
We adopt the convention that whenever we define a function on $I$, we automatically extend it $\Q$-linearly to $\II$.
We have therefore already defined, for example, $\zeta(-2(2,1)+3(4))$ as $-2\zeta(2,1)+3\zeta(4)$.

The definition immediately implies that the product of two multiple zeta values can be written as a $\Q$-linear combination of multiple zeta values; for instance
\begin{align*}
 \zeta(2)\zeta(3)
 &=\Biggl(\sum_{m=1}^{\infty}\frac{1}{m^2}\Biggr)\Biggl(\sum_{m=1}^{\infty}\frac{1}{m^3}\Biggr)
 =\sum_{m_1>m_2\ge1}\frac{1}{m_1^2m_2^3}+\sum_{m_1>m_2\ge1}\frac{1}{m_1^3m_2^2}+\sum_{m=1}^{\infty}\frac{1}{m^5}\\
 &=\zeta(2,3)+\zeta(3,2)+\zeta(5)=\zeta((2,3)+(3,2)+(5)).
\end{align*}
In order to capture this $\Q$-algebra structure of multiple zeta values,
we define a $\Q$-bilinear product $*$ on $\II$, known as the \emph{harmonic product} or \emph{stuffle product}, inductively by setting
\[
 \emptyset*(k_1,\dots,k_r)=(k_1,\dots,k_r)*\emptyset=(k_1,\dots,k_r)
\]
for all $(k_1,\dots,k_r)\in I$ and setting
\begin{align*}
 &(k_1,\dots,k_r)*(l_1,\dots,l_s)\\
 &=((k_1,\dots,k_r)*(l_1,\dots,l_{s-1}),l_s)+((k_1,\dots,k_{r-1})*(l_1,\dots,l_s),k_r)\\
 &\phantom{{}={}}+((k_1,\dots,k_{r-1})*(l_1,\dots,l_{s-1}),k_r+l_s)
\end{align*}
for all nonempty $(k_1,\dots,k_r),(l_1,\dots,l_s)\in I$.
The definition tells us that
\[
 (2)*(3)=((2)*\emptyset,3)+(\emptyset*(3),2)+(\emptyset*\emptyset,2+3)=(2,3)+(3,2)+(5),
\]
agreeing with the computation above.

Kontsevich observed (see \cite{Z}) that each multiple zeta value can be written as an iterated integral in the following fashion:
\begin{align*}
 &\zeta(k_1,\dots,k_r)\\
 &=\int_{1>t_1>\dots>t_{k_1+\dots+k_r}>0}\underbrace{\frac{dt_1}{t_1}\dotsm\frac{dt_{k_1-1}}{t_{k_1-1}}}_{k_1-1}\frac{dt_{k_1}}{1-t_{k_1}}\dotsm\underbrace{\frac{dt_{k_1+\dots+k_{r-1}+1}}{t_{k_1+\dots+k_{r-1}+1}}\dotsm\frac{dt_{k_1+\dots+k_r-1}}{t_{k_1+\dots+k_r-1}}}_{k_r-1}\frac{dt_{k_1+\dots+k_r}}{1-t_{k_1+\dots+k_r}}.
\end{align*}
The integral representation of multiple zeta values gives rise to another $\Q$-algebra structure, exemplified as below:
\begin{align*}
 \zeta(2)\zeta(3)
 &=\biggl(\int_{1>t_1>t_2>0}\frac{dt_1}{t_1}\frac{dt_2}{1-t_2}\biggr)\biggl(\int_{1>t_1>t_2>t_3>0}\frac{dt_1}{t_1}\frac{dt_2}{t_2}\frac{dt_3}{1-t_3}\biggr)\\
 &=\int_{1>t_1>t_2>t_3>t_4>t_5>0}\frac{dt_1}{t_1}\frac{dt_2}{1-t_2}\frac{dt_3}{t_3}\frac{dt_4}{t_4}\frac{dt_5}{1-t_5}\\
 &\phantom{{}={}}+3\int_{1>t_1>t_2>t_3>t_4>t_5>0}\frac{dt_1}{t_1}\frac{dt_2}{t_2}\frac{dt_3}{1-t_3}\frac{dt_4}{t_4}\frac{dt_5}{1-t_5}\\
 &\phantom{{}={}}+6\int_{1>t_1>t_2>t_3>t_4>t_5>0}\frac{dt_1}{t_1}\frac{dt_2}{t_2}\frac{dt_3}{t_3}\frac{dt_4}{1-t_4}\frac{dt_5}{1-t_5}\\
 &=\zeta(2,3)+3\zeta(3,2)+6\zeta(4,1)=\zeta((2,3)+3(3,2)+6(4,1)).
\end{align*}
This $\Q$-algebra structure is captured by the \emph{shuffle product} $\sha$, another $\Q$-bilinear product on $\II$ defined as follows.
We first associate to each index $\ve{k}=(k_1,\dots,k_r)$ a zero-one sequence
\[
 \varphi(\ve{k})=[\underbrace{0\dotsm0}_{k_1-1}1\dots\underbrace{0\dotsm0}_{k_r-1}1],
\]
and then define the shuffle product $\ve{k}\sha\ve{l}$ of indices $\ve{k}$ and $\ve{l}$
as the $\Q$-linear combination of indices corresponding to the $\Q$-linear combination of zero-one sequences
obtained by considering all ways of interleaving the two zero-one sequences $\varphi(\ve{k})$ and $\varphi(\ve{l})$.
For example, if we wish to find $(2)\sha(3)$, then we interleave $\varphi(2)=[01]$ and $\varphi(3)=[001]$
to obtain one $[01001]$, three $[00101]$s, and six $[00011]$s, so that
\[
 (2)\sha(3)=(2,3)+3(3,2)+6(4,1)
\]
agreeing with the computation above.

We therefore have two products $*$ and $\sha$ such that
\[
 \zeta(\ve{k})\zeta(\ve{l})=\zeta(\ve{k}*\ve{l})=\zeta(\ve{k}\sha\ve{l})
\]
for all admissible indices $\ve{k}$ and $\ve{l}$.

\subsection{Regularization theorem}
If $\ve{k}\in I$ is not admissible, then $\zeta(\ve{k})$ cannot be defined in the above-mentioned manner as the infinite sum diverges.
Ihara, Kaneko, and Zagier~\cite{IKZ} showed that we can uniquely define $\zeta^{*}(\ve{k};T),\zeta^{\sha}(\ve{k};T)\in\R[T]$, called the \emph{regularizations}, in such a way that
\begin{itemize}
 \item $\zeta^*(\ve{k};T)=\zeta^{\sha}(\ve{k};T)=\zeta(\ve{k})$ if $\ve{k}\in I$ is admissible;
 \item $\zeta^*(1;T)=\zeta^{\sha}(1;T)=T$;
 \item $\zeta^*(\ve{k};T)\zeta^*(\ve{l};T)=\zeta^*(\ve{k}*\ve{l};T)$ and $\zeta^{\sha}(\ve{k};T)\zeta^{\sha}(\ve{l};T)=\zeta^{\sha}(\ve{k}\sha\ve{l};T)$
       for all $\ve{k},\ve{l}\in I$.
\end{itemize}
They then proved the regularization theorem that describes the relationship between the two regularizations $\zeta^*(\ve{k};T)$ and $\zeta^{\sha}(\ve{k};T)$.
In order to state the theorem, we set
\[
 A(u)=\exp\Biggl(\sum_{n=2}^{\infty}\frac{(-1)^n}{n}\zeta(n)u^n\Biggr)\in\R[[u]]
\]
and define an $\R$-linear map $\rho\colon\R[T]\to\R[T]$ by
\[
 \rho(e^{Tu})=A(u)e^{Tu}
\]
in $\R[T][[u]]$ on which $\rho$ acts coefficientwise.

\begin{thm}[Regularization theorem, {\cite[Theorem~1]{IKZ}}]\label{thm:IKZ_reg}
 For $\ve{k}\in I$, we have
 \[
  \zeta^{\sha}(\ve{k};T)=\rho(\zeta^{*}(\ve{k};T)).
 \]
\end{thm}

\subsection{Statement of the main theorem}
Our main theorem is a polynomial generalization of the regularization theorem (Theorem~\ref{thm:IKZ_reg}).
The polynomial generalization of the multiple zeta values we shall be looking at is the following:
\begin{defn}
 For $\ve{k}=(k_1,\dots,k_r)\in I$ and $\bullet\in\{*,\sha\}$, we define
 \[
  \zeta_{x,y}^{\bullet}(\ve{k};T)=\sum_{i=0}^{r}x^{k_1+\dots+k_i}y^{k_{i+1}+\dots+k_r}\zeta^{\bullet}(k_i,\dots,k_1;T)\zeta^{\bullet}(k_{i+1},\dots,k_r;T)\in\R[x,y,T];
 \]
 we understand that $\zeta_{x,y}^{\bullet}(\emptyset;T)=1$.
\end{defn}

Note that $\zeta_{0,1}^{\bullet}(\ve{k};T)=\zeta^{\bullet}(\ve{k};T)$.

\begin{rem}
 Kaneko and Zagier~\cite{KZ} showed that $\zeta_{-1,1}^{*}(\ve{k};T)$ and $\zeta_{-1,1}^{\sha}(\ve{k};T)$
 are constants (i.e.\ independent of $T$) whose difference is $\zeta(2)$ times a $\Q$-linear combination of multiple zeta values,
 and called $\zeta_{-1,1}^{\bullet}(\ve{k};T)$ modulo $\zeta(2)$ the \emph{symmetric multiple zeta value}.
\end{rem}

\begin{rem}
 Although not used in this paper,
 it might be worthwhile to note that
 \[
  \zeta_{x,y}^{*}(\ve{k};T)\zeta_{x,y}^{*}(\ve{l};T)=\zeta_{x,y}^{*}(\ve{k}*\ve{l};T)
 \]
 for all $\ve{k},\ve{l}\in I$.
 Indeed, if $\ve{k}=(k_1,\dots,k_r)$ and $\ve{l}=(l_1,\dots,l_s)$, then
 \begin{align*}
  &\zeta_{x,y}^{*}(\ve{k};T)\zeta_{x,y}^{*}(\ve{l};T)\\
  &=\Biggl(\sum_{i=0}^{r}x^{k_1+\dots+k_i}y^{k_{i+1}+\dots+k_r}\zeta^{*}(k_i,\dots,k_1;T)\zeta^{*}(k_{i+1},\dots,k_r;T)\Biggr)\\
  &\phantom{{}={}}\times\Biggl(\sum_{j=0}^{s}x^{l_1+\dots+l_j}y^{l_{j+1}+\dots+l_s}\zeta^{*}(l_j,\dots,l_1;T)\zeta^{*}(l_{j+1},\dots,l_s;T)\Biggr)\\
  &=\sum_{i=0}^{r}\sum_{j=0}^{s}x^{k_1+\dots+k_i+l_1+\dots+l_j}y^{k_{i+1}+\dots+k_r+l_{j+1}+\dots+l_s}\\
  &\phantom{{}={}}\times\zeta^{*}(k_i,\dots,k_1;T)\zeta^{*}(k_{i+1},\dots,k_r;T)\zeta^{*}(l_j,\dots,l_1;T)\zeta^{*}(l_{j+1},\dots,l_s;T)\\
  &=\sum_{i=0}^{r}\sum_{j=0}^{s}x^{k_1+\dots+k_i+l_1+\dots+l_j}y^{k_{i+1}+\dots+k_r+l_{j+1}+\dots+l_s}\\
  &\phantom{{}={}}\times\zeta^{*}((k_i,\dots,k_1)*(l_j,\dots,l_1);T)\zeta^{*}((k_{i+1},\dots,k_r)*(l_{j+1},\dots,l_s);T)\\
  &=\zeta_{x,y}^{*}(\ve{k}*\ve{l};T).
 \end{align*}
 Here the last equality can be seen by observing that
 each summand of $\zeta_{x,y}^{*}(\ve{k}*\ve{l};T)$ comes from splitting, into two parts, an index that appears in the expansion of $\ve{k}*\ve{l}$
 and that such a summand can also be obtained by considering $(k_i,\dots,k_1)*(l_j,\dots,l_1)$ and $(k_{i+1},\dots,k_r)*(l_{j+1},\dots,l_s)$ for some $i$ and $j$.
 For example, if $r=5$ and $s=4$, then $\ve{k}*\ve{l}$ contains $(k_1,k_2+l_1,l_2,k_3,k_4+l_3,l_4,k_5)$ in its expansion
 and splitting it into $(k_1,k_2+l_1,l_2,k_3)$ and $(k_4+l_3,l_4,k_5)$ gives a summand
 \[
  x^{k_1+k_2+k_3+l_1+l_2}y^{k_4+k_5+l_3+l_4}\zeta^{*}(k_3,l_2,k_2+l_1,k_1;T)\zeta^{*}(k_4+l_3,l_4,k_5;T)
 \]
 in the expansion of $\zeta_{x,y}^{*}(\ve{k}*\ve{l};T)$.
 This summand can also be obtained by setting $i=3$ and $j=2$ and considering $(k_3,k_2,k_1)*(l_2,l_1)$ and $(k_4,k_5)*(l_3,l_4)$.
\end{rem}

We set
\[
 A_{x,y}(u)=\exp\Biggl(\sum_{n=2}^{\infty}\frac{(-1)^n}{n}\zeta(n)\frac{x^n+y^n}{(x+y)^n}u^n\Biggr)\in\R(x,y)[[u]]
\]
and define an $\R(x,y)$-linear map $\rho_{x,y}\colon\R(x,y)[T]\to\R(x,y)[T]$ by
\[
 \rho_{x,y}(e^{Tu})=A_{x,y}(u)e^{Tu}
\]
in $\R(x,y)[T][[u]]$ on which $\rho_{x,y}$ acts coefficientwise.
Note that $A_{0,1}(u)=A(u)$ and $\rho_{0,1}=\rho$.

\begin{thm}[Main theorem]\label{thm:main}
 For $\ve{k}\in I$, we have
 \[
  \zeta_{x,y}^{\sha}(\ve{k};T)=\rho_{x,y}(\zeta_{x,y}^{*}(\ve{k};T)).
 \]
\end{thm}

\section{Proof of the main theorem}
Each $\ve{k}\in I$ can be written uniquely as $\ve{k}=(\{1\}^b,\ve{l})$, where $b$ is a nonnegative integer and $\ve{l}$ is an admissible index;
we write $b(\ve{k})$ for this $b$ and put $\ve{k}^j=(\{1\}^{b-j},\ve{l})$ for $j=0,\dots,b$.

\begin{prop}\label{prop:reg_coeff}
 For $\ve{k}\in I$ and $\bullet\in\{*,\sha\}$, we have
 \[
  \zeta^{\bullet}(\ve{k};T)=\sum_{j=0}^{b(\ve{k})}\zeta^{\bullet}(\ve{k}^j;0)\frac{T^j}{j!}.
 \]
\end{prop}

\begin{proof}
 See \cite[Proposition~10]{IKZ}.
\end{proof}

For $\ve{k}=(k_1,\dots,k_r)\in I$, we define
\[
 w(\ve{k};X_1,X_2,\dots)=X_{k_1}\dotsm X_{k_r}\in\R\langle X_1,X_2,\dots\rangle.
\]
We further define
\begin{align*}
 \Phi^{\bullet}(T;X_1,X_2,\dots)&=\sum_{\ve{k}\in I}\zeta^{\bullet}(\ve{k};T)w(\ve{k};X_1,X_2,\dots)\in\R[T]\langle\langle X_1,X_2,\dots\rangle\rangle,\\
 \Phi_{x,y}^{\bullet}(T;X_1,X_2,\dots)&=\sum_{\ve{k}\in I}\zeta_{x,y}^{\bullet}(\ve{k};T)w(\ve{k};X_1,X_2,\dots)\in\R[x,y,T]\langle\langle X_1,X_2,\dots\rangle\rangle
\end{align*}
for $\bullet\in\{*,\sha\}$.

\begin{prop}\label{prop:isolate_T}
 For $\bullet\in\{*,\sha\}$, we have
 \[
  \Phi^{\bullet}(T;X_1,X_2,\dots)=e^{TX_1}\Phi^{\bullet}(0;X_1,X_2,\dots).
 \]
\end{prop}

\begin{proof}
 Proposition~\ref{prop:reg_coeff} shows that
 \begin{align*}
  \Phi^{\bullet}(T;X_1,X_2,\dots)&=\sum_{\ve{k}\in I}\zeta^{\bullet}(\ve{k};T)w(\ve{k};X_1,X_2,\dots)\\
  &=\sum_{\ve{k}\in I}\Biggl(\sum_{j=0}^{b(\ve{k})}\zeta^{\bullet}(\ve{k}^j;0)\frac{T^j}{j!}\Biggr)w(\ve{k};X_1,X_2,\dots)\\
  &=\sum_{\ve{k}\in I}\sum_{j=0}^{b(\ve{k})}\zeta^{\bullet}(\ve{k}^j;0)\frac{T^j}{j!}X_1^jw(\ve{k}^j;X_1,X_2,\dots)\\
  &=\sum_{j=0}^{\infty}\sum_{\ve{l}\in I}\zeta^{\bullet}(\ve{l};0)\frac{T^j}{j!}X_1^jw(\ve{l};X_1,X_2,\dots)\\
  &=\Biggl(\sum_{j=0}^{\infty}\frac{T^jX_1^j}{j!}\Biggr)\Biggl(\sum_{\ve{l}\in I}\zeta^{\bullet}(\ve{l};0)w(\ve{l};X_1,X_2,\dots)\Biggr)\\
  &=e^{TX_1}\Phi^{\bullet}(0;X_1,X_2,\dots).\qedhere
 \end{align*}
\end{proof}

\begin{prop}\label{prop:sha_*}
 We have
 \[
  \Phi^{\sha}(T;X_1,X_2,\dots)=A(X_1)e^{TX_1}\Phi^{*}(0;X_1,X_2,\dots).
 \]
\end{prop}

\begin{proof}
 If we extend $\rho\colon\R[T]\to\R[T]$ coefficientwise to a map from $\R[T]\langle\langle X_1,X_2,\dots\rangle\rangle$ to itself,
 then the regularization theorem (Theorem~\ref{thm:IKZ_reg}) shows that
 \[
  \Phi^{\sha}(T;X_1,X_2,\dots)=\rho(\Phi^{*}(T;X_1,X_2,\dots)).
 \]
 Therefore it follows from Proposition~\ref{prop:isolate_T} that
 \begin{align*}
  \Phi^{\sha}(T;X_1,X_2,\dots)
  &=\rho(\Phi^{*}(T;X_1,X_2,\dots))\\
  &=\rho(e^{TX_1}\Phi^{*}(0;X_1,X_2,\dots))\\
  &=\rho(e^{TX_1})\Phi^{*}(0;X_1,X_2,\dots)\\
  &=A(X_1)e^{TX_1}\Phi^{*}(0;X_1,X_2,\dots).\qedhere
 \end{align*}
\end{proof}

Define an $\R[x,y,T]$-linear map $\alpha$ from $\R[x,y,T]\langle\langle X_1,X_2,\dots\rangle\rangle$ to itself
by setting $\alpha(X_{k_1}\dotsm X_{k_r})=X_{k_r}\dotsm X_{k_1}$ for $\ve{k}=(k_1,\dots,k_r)\in I$.

\begin{prop}\label{prop:isolate_x}
 For $\bullet\in\{*,\sha\}$, we have
 \[
  \Phi_{x,y}^{\bullet}(T;X_1,X_2,\dots)=\alpha(\Phi^{\bullet}(T;x^1X_1,x^2X_2,\dots))\Phi^{\bullet}(T;y^1X_1,y^2X_2,\dots).
 \]
\end{prop}

\begin{proof}
 We have
 \begin{align*}
  &\Phi_{x,y}^{\bullet}(T;X_1,X_2,\dots)\\
  &=\sum_{\ve{k}\in I}\zeta_{x,y}^{\bullet}(\ve{k};T)w(\ve{k};X_1,X_2,\dots)\\
  &=\sum_{r=0}^{\infty}\sum_{k_1,\dots,k_r=1}^{\infty}\sum_{i=0}^{r}x^{k_1+\dots+k_i}y^{k_{i+1}+\dots+k_r}\zeta^{\bullet}(k_i,\dots,k_1;T)\zeta^{\bullet}(k_{i+1},\dots,k_r;T)X_{k_1}\dotsm X_{k_r}\\
  &=\Biggl(\sum_{i=0}^{\infty}\sum_{k_1,\dots,k_i=1}^{\infty}\zeta^{\bullet}(k_i,\dots,k_1;T)(x^{k_1}X_{k_1})\dotsm(x^{k_i}X_{k_i})\Biggr)\\
  &\phantom{{}={}}\times\Biggl(\sum_{j=0}^{\infty}\sum_{l_1,\dots,l_j=1}^{\infty}\zeta^{\bullet}(l_1,\dots,l_j;T)(y^{l_1}X_{l_1})\dotsm(y^{l_j}X_{l_j})\Biggr)\\
  &=\alpha(\Phi^{\bullet}(T;x^1X_1,x^2X_2,\dots))\Phi^{\bullet}(T;y^1X_1,y^2X_2,\dots).\qedhere
 \end{align*}
\end{proof}

\begin{prop}\label{prop:key}
 We have
 \begin{align*}
  \Phi_{x,y}^{*}(T;X_1,X_2,\dots)&=\alpha(\Phi^{*}(0;x^1X_1,x^2X_2,\dots))e^{(x+y)TX_1}\Phi^{*}(0;y^1X_1,y^2X_2,\dots),\\
  \Phi_{x,y}^{\sha}(T;X_1,X_2,\dots)&=\alpha(\Phi^{*}(0;x^1X_1,x^2X_2,\dots))A_x((x+y)X_1)e^{(x+y)TX_1}\Phi^{*}(0;y^1X_1,y^2X_2,\dots).
 \end{align*}
\end{prop}

\begin{proof}
 Propositions \ref{prop:isolate_T} and \ref{prop:isolate_x} show that
 \begin{align*}
  \Phi_{x,y}^{*}(T;X_1,X_2,\dots)
  &=\alpha(\Phi^{*}(T;x^1X_1,x^2X_2,\dots))\Phi^{*}(T;y^1X_1,y^2X_2,\dots)\\
  &=\alpha(e^{xTX_1}\Phi^{*}(0;x^1X_1,x^2X_2,\dots))e^{yTX_1}\Phi^{*}(0;y^1X_1,y^2X_2,\dots)\\
  &=\alpha(\Phi^{*}(0;x^1X_1,x^2X_2,\dots))e^{xTX_1}e^{yTX_1}\Phi^{*}(0;y^1X_1,y^2X_2,\dots)\\
  &=\alpha(\Phi^{*}(0;x^1X_1,x^2X_2,\dots))e^{(x+y)TX_1}\Phi^{*}(0;y^1X_1,y^2X_2,\dots),
 \end{align*}
 which is the first desired equality.

 In a similar manner, Propositions \ref{prop:sha_*} and \ref{prop:isolate_x} show that
 \begin{align*}
  &\Phi_{x,y}^{\sha}(T;X_1,X_2,\dots)\\
  &=\alpha(\Phi^{\sha}(T;x^1X_1,x^2X_2,\dots))\Phi^{\sha}(T;y^1X_1,y^2X_2,\dots)\\
  &=\alpha(A(xX_1)e^{xTX_1}\Phi^{*}(0;x^1X_1,x^2X_2,\dots))A(yX_1)e^{yTX_1}\Phi^{*}(0;y^1X_1,y^2X_2,\dots)\\
  &=\alpha(\Phi^{*}(0;x^1X_1,x^2X_2,\dots))e^{xTX_1}A(xX_1)A(yX_1)e^{yTX_1}\Phi^{*}(0;y^1X_1,y^2X_2,\dots).
 \end{align*}
 Since $e^{xTX_1}$, $A(xX_1)$, $A(yX_1)$, and $e^{yTX_1}$ are commutative with each other and since
 \begin{align*}
  A(xX_1)A(yX_1)
  &=\exp\Biggl(\sum_{n=2}^{\infty}\frac{(-1)^n}{n}\zeta(n)x^nX_1^n\Biggr)\exp\Biggl(\sum_{n=2}^{\infty}\frac{(-1)^n}{n}\zeta(n)y^nX_1^n\Biggr)\\
  &=\exp\Biggl(\sum_{n=2}^{\infty}\frac{(-1)^n}{n}\zeta(n)(x^n+y^n)X_1^n\Biggr)\\
  &=A_{x,y}((x+y)X_1),
 \end{align*}
 we have proved the second desired equality.
\end{proof}

\begin{thm}\label{thm:main_gen_func}
 We have
 \[
  \Phi_{x,y}^{\sha}(T;X_1,X_2,\dots)=\rho_{x,y}(\Phi_{x,y}^{*}(T;X_1,X_2,\dots)).
 \]
\end{thm}

\begin{proof}
 Proposition~\ref{prop:key} shows that
 \begin{align*}
  &\Phi_{x,y}^{\sha}(T;X_1,X_2,\dots)\\
  &=\alpha(\Phi^{*}(0;x^1X_1,x^2X_2,\dots))A_{x,y}((x+y)X_1)e^{(x+y)TX_1}\Phi^{*}(0;y^1X_1,y^2X_2,\dots)\\
  &=\alpha(\Phi^{*}(0;x^1X_1,x^2X_2,\dots))\rho_{x,y}(e^{(x+y)TX_1})\Phi^{*}(0;y^1X_1,y^2X_2,\dots)\\
  &=\rho_{x,y}(\alpha(\Phi^{*}(0;x^1X_1,x^2X_2,\dots))e^{(x+y)TX_1}\Phi^{*}(0;y^1X_1,y^2X_2,\dots))\\
  &=\rho_{x,y}(\Phi_{x,y}^{*}(T;X_1,X_2,\dots)).\qedhere
 \end{align*}
\end{proof}

Our main theorem (Theorem~\ref{thm:main}) immediately follows from Theorem~\ref{thm:main_gen_func}.

\section*{Acknowledgements}
This work was supported by JSPS KAKENHI Grant Numbers JP18J00982, JP18K03243, and JP18K13392.
The authors wish to express their gratitude to Yoshinori Yamasaki for helpful comments.


\begin{thebibliography}{99}
 \bibitem{IKZ}
  K. Ihara, M. Kaneko, and D. Zagier,
  \emph{Derivation and double shuffle relations for multiple zeta values},
  Compos. Math. \textbf{142} (2006), 307--338.
 \bibitem{KZ}
  M. Kaneko and D. Zagier,
  \emph{Finite multiple zeta values},
  in preparation.
 \bibitem{Z}
  D. Zagier,
  \emph{Values of zeta functions and their applications},
  First European Congress of Mathematics, Vol. II (Paris, 1992),
  Progr. Math., vol. 120, Birkh\"auser, Basel, 1994, pp.~497--512.
\end{thebibliography}
\end{document}